\documentclass[11pt]{amsart}

\usepackage{latexsym}
\usepackage{amsfonts}

\usepackage{pgf,tikz}
\usepackage{mathrsfs}
\usetikzlibrary{arrows}
\usepackage{color}
\usepackage{float}
\usepackage{subfigure}
\usepackage{caption}

\newcommand{\field}[1]{\mathbb{#1}}
\newcommand{\C}{\field{C}}

\newcommand{\R}{\field{R}}

\DeclareMathOperator{\rank}{rank}
\DeclareMathOperator{\Stab}{Stab}
\DeclareMathOperator{\im}{Im}
\DeclareMathOperator{\sgn}{sgn}

\newtheorem{defi}{Definition}[section]

\newtheorem{lem}[defi]{Lemma}
\newtheorem{theo}[defi]{Theorem}

\newtheorem{re}[defi]{Remark}

\subjclass{14 D 99, 14 R 99, 51 M 99}
\thanks{The  authors were partially supported by the Narodowe Centrum Nauki grant number 2015/17/B/ST1/02637.}
\title[On quadratic polynomial mappings of $f:\C^2\to\C^2$]{On quadratic polynomial mappings  $f:\C^2\to\C^2$} \makeatletter

\author{M. Farnik \& Z. Jelonek}

\address[M. Farnik]{Jagiellonian University\\
Faculty of Mathematics and Computer Science\\
{\L}ojasiewicza 6, 30-348 Krak\'ow, Poland}
\email{michal.farnik@gmail.com}

\address[Z. Jelonek]{Instytut Matematyczny\\
Polska Akademia Nauk\\
\'Sniadeckich 8, 00-656 Warszawa, Poland}
\email{najelone@cyf-kr.edu.pl}

\date{\today}

\begin{document}

\begin{abstract}{We show that up to linear equivalence, there is only finitely many  polynomial
quadratic mappings $f:\C^2\to\C^2$ and $f:\R^2\to\R^2$. We list all possibilities.}
\end{abstract}

\maketitle

\section{Introduction}
Let $\Omega(d_1,d_2)$ denote the space of polynomial mappings $f=(f_1,f_2):\C^2\to\C^2$, where deg $f_1\le d_1$ and
deg $f_2\le d_2$. Let $f,g\in \Omega(d_1,d_2)$. We say that $f$ is topologically (respectively linearly) equivalent to $g$, if there are homeomorphisms (respectively linear isomorphisms) $\Phi,\Psi:\C^2\to\C^2$, such that $f=\Psi\circ g\circ \Phi$. In the paper \cite{a-n} it was showed that there is only a finite number of different topological types of mappings in  $\Omega(d_1,d_2)$. Moreover, we know (see e.g. \cite{jel2}) that there is a Zariski open dense subset $U\subset \Omega(d_1,d_2)$ such that every mapping $f\in U$ has the same, generic, topological type. If a mapping $f$  has a generic topological type, then we say that $f$ is a generic mapping.  In practice it is difficult to describe the generic  topological type and other topological types  effectively.

Here we consider the case $d_1=d_2=2.$ We show that in this case the topological equivalence almost coincides with the linear equivalence. Moreover, we obtain a full classification of quadratic mappings of $\C^2,$ with respect to the linear (hence also topological) equivalence. In particular we find a model of a generic mapping of $\Omega(2,2).$

We explain our basic idea. It was proved in \cite{fjr} that a generic polynomial mapping from $\Omega(2,2)$ has a rational cuspidial curve of degree $4$ as a discriminant, moreover this curve is tangent in two smooth points to the line at infinity. On the other hand any two such rational cuspidial curves are projectively equivalent (see \cite{moe}), hence it is easy to deduce that   discriminants of generic mappings from $\Omega(2,2)$ are linearly equivalent. Using \cite{jel} we can deduce that any two generic mappings from $\Omega(2,2)$ are algebraically equivalent and consequently (since they have the same algebraic degree) they are linearly equivalent. Moreover,  there are only finitely many  possible discriminants of quadratic mappings, up to a linear equivalence (because in the non-generic case they have a non-trivial action of an infinite affine group). Hence we  can expect that there is only a finite number of orbits of the action of a linear group on $\Omega(2,2)$. We show that it is indeed the case. In fact, it can be done in a very elementary way, however it is suprising that this result was not discovered up till now.

Let $C(f), \Delta(f), \mu(f)$ denote : the set of critical points, the discriminant and topological degree (see Definition \ref{defi}). We have the following possibilities:

Generically-finite mappings:

(1) (the generic case) $f_1=(x^2+y,y^2+x)$, $C(f_1)=\{4xy-1=0\}$ is a hyperbola and $\Delta(f_1)=\{2^8x^2y^2-2^8x^3-2^8y^3+2^5\cdot 9xy-27=0\}$ is a reduced and irreducible curve with $3$ cusps at points $f_1(\frac{\varepsilon}{2},\frac{\varepsilon^2}{2})$, where $\varepsilon^3=1$, $\dim O(f_1)=12$,
$O(f_1)$ is an open and dense affine subvariety of $\Omega(2,2)$, moreover $\chi(O(f_1))=0$ and $\mu(f_1)=4$.

(2) $f_2=(x^2+y,xy)$, $C(f_2)=\{2x^2=y\}$ is a parabola and $\Delta(f_2)=\{4x^3=27y^2\}$ is a cusp, $\dim O(f_2)=11$,
$O(f_2)$ is an affine subvariety of $\Omega(2,2)$, $\mu(f_2)=3$.

(3) $f_3=(x^2+y, y^2)$, $C(f_3)=\{4xy=0\}$ is two intersecting lines and $\Delta(f_3)=\{y(y-x^2)=0\}$ is the sum of a line and a parabola,
$\dim O(f_3)=11$, $\mu(f_3)=4$.

(4) $f_4=(x^2,y^2)$, $C(f_4)=\{4xy=0\}$ is two intersecting lines and $\Delta(f_4)=\{xy=0\}$ is also  two intersecting lines, $\dim O(f_4)=10$ and $\mu(f_4)=4$.

(5) $f_5=(x^2-x, xy)$, $C(f_5)=\{2x^2-x=0\}$ is two parallel lines and $\Delta(f_5)$ is the sum of the line $x=-1/4$ and the line $x=0$,
$\dim O(f_5)=10$, $\mu(f_5)=2$ and $f_5$ is not proper.

(6) $f_6=(x^2,xy)$, $C(f_6)=\{x^2=0\}$ is a double line and $\Delta(f_6)$ is the line $x=0$, $\dim O(f_6)=9$, $\mu(f_6)=2$ and $f_6$ is not proper.

(7) $f_7=(xy,x+y)$, $C(f_7)=\{y=x\}$ is a line and $\Delta(f_7)=\{4x-y^2=0\}$ is a parabola, $\dim O(f_7)=10$ and $\mu(f_7)=2$.

(7') $f_9=(x^2,y)$, $C(f_9)=\{x=0\}$ is a line and $\Delta(f_9)=\{x=0\}$ is also a line, $\dim O(f_9)=9$ and $\mu(f_9)=2$.

(8) $f_8=(x,xy)$, $C(f_8)=\{x=0\}$ is a line and $\Delta(f_8)$ is the line $\{ x=0\}$, $\dim O(f_8)=9$, $\mu(f_8)=1$ and $f_8$ is not proper.

(9) $f_{10}=(x^2+y,x)$, $\dim O(f_{10})=8$, $C(f_{10})$ is the empty set and $f_{10}$ is an automorphism, $\mu(f_{10})=1$.

(9') $f_{12}=(x,y)$, $\dim O(f_{12})=6$, $C(f_{12})$ is the empty set and $f_{12}$ is an automorphism, $\mu(f_{12})=1$.
\vspace{5mm}

Not generically-finite mappings:
\vspace{2mm}

(10) $f_{11}=(x^2,x)$, $\dim O(f_{11})=7$ and $C(f_{11})$ is the plane.

(10') $f_{14}=(x^2+y,0)$, $\dim O(f_{14})=7$ and $C(f_{14})$ is the plane.

(10'') $f_{16}=(x,0)$, $\dim O(f_{16})=5$ and $C(f_{16})$ is the plane.

(11) $f_{13}=(xy,0)$, $\dim O(f_{13})=8$ and $C(f_{13})$ is the plane.

(12) $f_{15}=(x^2,0)$, $\dim O(f_{15})=6$ and $C(f_{15})$ is the plane.

(13) $f_{17}=(0,0)$, $\dim O(f_{17})=2$ and $C(f_{17})$ is the plane.

\vspace{5mm}

Note that mappings $f_{7}$ and $f_{9}$ are topologically (even algebraically) equivalent. Similarly mappings $f_{10}, f_{12}$ and $f_{11}, f_{14}, f_{16}$  Hence we have $13$ different topological types and $17$ different linear types.

In Figure \ref{Picture1} we present the structure of $\Omega(2,2)$. Each row consists of orbits of given dimension, from the largest to the smallest, and a rising path joins two orbits if the smaller is contained in the closure of the larger.

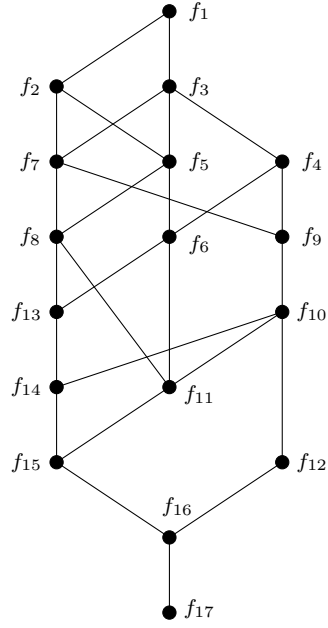
\begin{figure}[H]
                \centering
\begin{tikzpicture}
\clip(-0.16,-0.1) rectangle (4.6,8.3);
\draw (0.5,7.)-- (0.5,2.);
\draw (2.,8.)-- (2.,5.);
\draw (3.5,6.)-- (3.5,2.);
\draw (2.,1.)-- (2.,0.);
\draw (2.,5.)-- (2.,3.);
\draw (2.,8.)-- (0.5,7.);
\draw (0.5,7.)-- (2.,6.);
\draw (2.,7.)-- (0.5,6.);
\draw (2.,7.)-- (3.5,6.);
\draw (0.5,6.)-- (3.5,5.);
\draw (2.,6.)-- (0.5,5.);
\draw (3.5,6.)-- (2.,5.);
\draw (2.,5.)-- (0.5,4.);
\draw (3.5,4.)-- (2.,3.);
\draw (2.,3.)-- (0.5,2.);
\draw (0.5,2.)-- (2.,1.);
\draw (3.5,2.)-- (2.,1.);
\draw (3.5,4.)-- (0.5,3.);
\draw (0.5,5.)-- (2.,3.);
\begin{scriptsize}
\draw [fill=black] (0.5,7.) circle (2.5pt);
\draw (0.15,7.) node {$f_2$};
\draw [fill=black] (0.5,2.) circle (2.5pt);
\draw (0.1,2.) node {$f_{15}$};
\draw [fill=black] (0.5,3.) circle (2.5pt);
\draw (0.1,3.) node {$f_{14}$};
\draw [fill=black] (0.5,4.) circle (2.5pt);
\draw (0.1,4.) node {$f_{13}$};
\draw [fill=black] (0.5,5.) circle (2.5pt);
\draw (0.15,5.) node {$f_8$};
\draw [fill=black] (0.5,6.) circle (2.5pt);
\draw (0.15,6.) node {$f_7$};
\draw [fill=black] (2.,8.) circle (2.5pt);
\draw (2.4,8.) node {$f_1$};
\draw [fill=black] (2.,5.) circle (2.5pt);
\draw (2.4,4.9) node {$f_6$};
\draw [fill=black] (3.5,6.) circle (2.5pt);
\draw (3.9,6.) node {$f_4$};
\draw [fill=black] (3.5,2.) circle (2.5pt);
\draw (3.9,2) node {$f_{12}$};
\draw [fill=black] (2.,7.) circle (2.5pt);
\draw (2.4,7.) node {$f_3$};
\draw [fill=black] (2.,6.) circle (2.5pt);
\draw (2.4,6.) node {$f_5$};
\draw [fill=black] (3.5,5.) circle (2.5pt);
\draw (3.9,5) node {$f_9$};
\draw [fill=black] (3.5,4.) circle (2.5pt);
\draw (3.9,4.) node {$f_{10}$};
\draw [fill=black] (2.,0.) circle (2.5pt);
\draw (2.4,0.05) node {$f_{17}$};
\draw [fill=black] (2.,1.) circle (2.5pt);
\draw (2.109,1.4075) node {$f_{16}$};
\draw [fill=black] (2.,3.) circle (2.5pt);
\draw (2.4,2.9) node {$f_{11}$};
\end{scriptsize}
\end{tikzpicture}
                \caption{The orbits\label{Picture1}}
                \end{figure}

We also obtain similar results in the real case. We have the following possibilities (here $\Delta(f)=f(C(f)))$:

Generically-finite mappings:

(1) (the first generic case) $f_1=(x^2+y,y^2+x)$, $C(f_1)=\{4xy-1=0\}$ is a hyperbola and $\Delta(f_1)=\{2^8x^2y^2-2^8x^3-2^8y^3+2^5\cdot 9xy-27=0\}$ is a reduced and irreducible curve with a cusp at $f_1(\frac{1}{2},\frac{1}{2})$, $\dim O(f_1)=12$,
$O(f_1)$ is an open semi-algebraic subvariety of $\Omega(2,2)$.

(1a) (the second generic case)  $f_{1'}=\left(x^2-y^2+x,2xy-y\right)$ with $C(f_{1'})=\left\{4x^2+4y^2-1=0\right\}$ a circle and $\Delta(f_{1'})=\{2^{16}(x^2+y^2)^2+2^{17}(-x^3+3xy^2)+2^9\cdot 3^3\cdot 5(x^2+y^2)-(15)^3\}$ a reduced and irreducible curve with $3$ cusps at points $f_{1'}\left(\frac{1}{4},\frac{-\sqrt{3}}{4}\right)$, $f_{1'}\left(\frac{1}{4},\frac{\sqrt{3}}{4}\right)$ and $f_{1'}\left(\frac{1}{2},0\right)$, $\dim O(f_{1'})=12$,
$O(f_{1'})$ is an open semi-algebraic subvariety of $\Omega(2,2)$.

(2) $f_2=(x^2+y,xy)$, $C(f_2)=\{2x^2=y\}$ is a parabola and $\Delta(f_2)=\{4x^3=27y^2\}$ is a cusp, $\dim O(f_2)=11$,
$O(f_2)$ is a semi-algebraic subvariety of $\Omega(2,2)$, $\mu(f_2)=3$.

(3) $f_3=(x^2+y, y^2)$, $C(f_3)=\{4xy=0\}$ is two intersecting lines and $\Delta(f_3)=\{y(y-x^2)=0\}$ is the sum of a line and a parabola,
$\dim O(f_3)=11$.

the restriction of complex orbit of $f_4=\left(x^2,y^2\right)$ splits into two orbits given by $\overline{\Phi}_1(f)<0$ and $\overline{\Phi}_1(f)>0$:

(4)  the first case is represented by $f_4$ with $C(f_4)=\{4xy=0\}$ two intersecting lines and $\Delta(f_4)=\{xy=0\}$ also two intersecting lines, $\dim O(f_4)=10$.

(4a) the second case is represented by $f_{4'}=\left(x^2-y^2,xy\right)$ with $C(f_{4'})=\{(0,0)\}$ a point, $\Delta(f_{4'})=\{(0,0)\}$- a point, $\dim O(f_{4'})=10$.

(5) $f_5=(x^2-x, xy)$, $C(f_5)=\{2x^2-x=0\}$ is two parallel lines and $\Delta(f_5)$ is the sum of the line $x=-1/4$ and the point $(0,0)$,
$\dim O(f_5)=10$, $\mu(f_5)=2$ and $f_5$ is not proper.

(6) $f_6=(x^2,xy)$, $C(f_6)=\{x^2=0\}$ is a double line and $\Delta(f_6)$ is the point $(0,0)$, $\dim O(f_6)=9$, $\mu(f_6)=2$ and $f_6$ is not proper.

(7) $f_7=(xy,x+y)$, $C(f_7)=\{y=x\}$ is a line and $\Delta(f_7)=\{4x-y^2=0\}$ is a parabola, $\dim O(f_7)=10$ and $\mu(f_7)=2$.

(7') $f_9=(x^2,y)$, $C(f_9)=\{x=0\}$ is a line and $\Delta(f_9)=\{x=0\}$ is also a line, $\dim O(f_9)=9$ and $\mu(f_9)=2$.

(8) $f_8=(x,xy)$, $C(f_8)=\{x=0\}$ is a line and $\Delta(f_8)$ is the point $(0,0)$, $\dim O(f_8)=9$, $\mu(f_8)=1$ and $f_8$ is not proper.

(9) $f_{10}=(x^2+y,x)$, $\dim O(f_{10})=8$, $C(f_{10})$ is the empty set and $f_{10}$ is an automorphism, $\mu(f_{10})=1$.

(9') $f_{12}=(x,y)$, $\dim O(f_{12})=6$, $C(f_{12})$ is the empty set and $f_{12}$ is an automorphism, $\mu(f_{12})=1$.
\vspace{5mm}

Not generically-finite mappings:
\vspace{2mm}

(10) $f_{11}=(x^2,x)$, $\dim O(f_{11})=7$ and $C(f_{11})$ is the plane.

(10') $f_{14}=(x^2+y,0)$, $\dim O(f_{14})=7$ and $C(f_{14})$ is the plane.

(10'') $f_{16}=(x,0)$, $\dim O(f_{16})=5$ and $C(f_{16})$ is the plane.

(11) $f_{13}=(xy,0)$, $\dim O(f_{13})=8$ and $C(f_{13})$ is the plane.

(12) $f_{15}=(x^2,0)$, $\dim O(f_{15})=6$ and $C(f_{15})$ is the plane.

(13) $f_{17}=(0,0)$, $\dim O(f_{17})=2$ and $C(f_{17})$ is the plane.

\vspace{5mm}

Thus we have $15$ different topological real types and $19$ different linear real types.

\section{Main Result}\label{secMR}

Let us recall the following:

\begin{defi}\label{defi}
If $f:\C^n\to\C^n$ is a generically-finite regular
mapping, then $\Delta(f)=\{ x\in \C^n : \# f^{-1}(x)\not=\mu(f)\}$ is
called the discriminant of $f$. In particular, if $f$ is proper, then
$\Delta(f)=f(C(f))$, where $C(f)$ is the critical set of $f$. The set $\Delta(f)$ is either a hypersurface or the empty set
(see e.g. \cite{j-k}).
\end{defi}

We also use the following:

\begin{defi}
By $GA(2,2)$ we denote the group of affine transformations of $\C^2$. By ${\mathcal GA(2,2)}$ we denote the group
$GA(2,2)\times GA(2,2)$ with multiplication $(L_1,R_1)\circ (L_2,R_2)=(L_1L_2, R_2R_1)$. The group ${\mathcal GA(2,2)}$
acts on the set $\Omega(2,2)$: $(L,R)f=L\circ f\circ R$. We denote the orbit of $f\in\Omega(2,2)$ by $O(f)$. We say that $f_1,f_2\in \Omega(2,2)$ are linearly equivalent, if there is a $\alpha\in {\mathcal GA(2,2)}$ such that $f_2=\alpha f_1$, i.e. $f_2\in O(f_1)$.
\end{defi}

We denote by $\Stab(f)$ the stabilizer of $f$. Note that $\dim\Stab(f)+\dim O(f)=\dim{\mathcal GA(2,2)}=12$. Moreover if $(L,R)\in\Stab(f)$ and the image of $f$ is not contained in a linear subspace of $\C^2$ then $R$ uniquely determines $L$.

Let $f_1=(x^2+y,y^2+x)$. We have:

\begin{theo}\label{th1}
Let $f\in \Omega(2,2)$. The following conditions are equivalent:

\begin{enumerate}
\item $C(f)$ is a hyperbola.
\item $O(f)$ is an open dense subset of $\Omega(2,2)$.
\item $O(f)=O(f_1)$, i.e. $f$ is linearly equivalent to $f_1$.
\end{enumerate}
\noindent Moreover, the maximal orbit $O(f_1)$ is a smooth affine variety of Euler characteristic $0$ and dimension $12$.
\end{theo}
\begin{proof}

(1) $\Rightarrow$ (2) First note that either the mapping $f|_{C(f)}$ has a critical point or $\Delta(f)$ is a singular curve.
Indeed, in other case $\Delta(f)$ is isomorphic to a hyperbola and the curve $f^{-1}(\Delta(f))$ has the Euler characteristic equal to $0$ (as a covering of hyperbola). In particular $\mu(f)\chi(\C^2\setminus\Delta(f))=\chi(\C^2\setminus  f^{-1}(\Delta(f)))=1$ and $\mu(f)=1$, a contradiction.

Now we will show that $\Stab(f)$ is finite. If $(L,R)\in\Stab(f)$ then $R$ must preserve $C(f)$, all singular points of $f|_{C(f)}$ and all pre-images of singular points of $\Delta(f)$. Hence $R$ has to preserve a hyperbola and some non-empty finite subset of the hyperbola. It is easy to see that there is only a finite number of such linear mappings $R$.

Since $\Stab(f)$ is finite we know that $\dim O(f)=\dim\Omega(2,2)$ and $O(f)$ is a maximal orbit, in particular it is open and dense.

(2) $\Rightarrow$ (3) It is easy to check that $C(f_1)$ is a hyperbola, hence ${\mathcal GA(2,2)}f_1$ is open and dense. In particular $O(f)\cap O(f_1)\not=\emptyset$.

(3) $\Rightarrow$ (1) Obvious.

Now we will describe the set $\Stab(f_1)$. Let us note that $C(f_1)=\{(x,y):4xy=1\}.$ Hence $f_1$ has exactly three $A_{1,1}$ singularities $P_i=(-1/2\alpha, -1/2\alpha^2)$, where $\alpha^3=1.$ Note that the linear automorphisms which preserve $C(f_1)$
and $P_i$ for $i=1,2,3$ are exactly the mappings $(x,y), (\alpha x,\alpha^2 y), (\alpha^2x, \alpha y)$ and $(y,x), (\alpha y,\alpha^2 x), (\alpha^2 y, \alpha x)$. Hence $\Stab(f_1)$ has exactly $6$ elements. In order to calculate the Euler characteristic of the maximal orbit $O={\mathcal GA(2,2)}f_1$ note that the mapping $\Psi\colon{\mathcal GA(2,2)}\ni(L,R)\mapsto L\circ f_1\circ R\in O$ is a topological covering of degree $6$. Hence $\chi({\mathcal GA(2,2)})=6\chi(O)$. Since $GA(2,2)$ is homotopic to $\C^4$ minus a cone, it has Euler characteristic equal to $0$. Since ${\mathcal GA(2,2)}=GA(2,2)\times GA(2,2)$ we have $\chi({\mathcal GA(2,2)})=0$. Hence $\chi(O)=0$.
\end{proof}

\begin{re}
{\rm We say that $f\in \Omega(2,2)$ is general if $O(f)$ is dense in $\Omega(2,2)$. Hence a general mapping is equivalent to $(x^2+y, y^2+x)$. Now assume that $f\in \Omega(2,2)$ is not a general mapping. From the proof of Theorem \ref{th1} we conclude that there are infinite algebraic groups of affine mappings of $\C^2$, which preserve $C(f)$, $\Delta(f)$ and $f^{-1}(\Delta(f))$. Hence the curves  $C(f)$, $\Delta(f)$, $f^{-1}(\Delta(f))$ are very special. In particular, the curve $f^{-1}(\Delta(f))$ is rational. }
\end{re}

\begin{re}
{\rm From our classification we have that a generic mapping is general. In other words a generic mapping  is linearly equivalent to $(x^2+y,y^2+x)$. }
\end{re}

Now we will describe the other orbits.
First we will show that there are exactly two orbits of dimension $11$. Let $f\in \Omega(2,2)$, we can write $f=(g,h)$ where
$$g=a_1x^2+b_1xy+c_1y^2+d_1x+e_1y+f_1,$$ $$h=a_2x^2+b_2xy+c_2y^2+d_2x+e_2y+f_2.$$
Let $A=2a_1b_2-2a_2b_1$, $B=4a_1c_2-4a_2c_1$, $C=2b_1c_2-2b_2c_1$, $D=2a_1e_2+d_1b_2-2a_2e_1-d_2b_1$, $E=2d_1c_2+b_1e_2-2d_2c_1-b_2e_1$ and $F=d_1e_2-d_2e_1$, the Jacobian of $f$ is given by the formula:
$$J(f)=Ax^2+Bxy+Cy^2+Dx+Ey+F.$$
Consider the following two matrices:

\vspace{5mm}

$$\Phi_1(f)= \left[\begin{matrix} 2A & B \\ B & 2C \end{matrix}\right] \quad\text{and}\quad
\Phi_2(f)= \left[\begin{matrix} 2A & B & D\\ B & 2C & E \\
D & E & 2F\end{matrix}\right]$$

Note that polynomial $\overline{\Phi}_1:=\det\Phi_1$ is a quadratic function of $c_1$ and $\overline{\Phi}_2:=\det\Phi_2$ is a quadratic function of $d_1$.
Using the obvious lemma below it is easy to check that the polynomials $\overline{\Phi}_1$ and $\overline{\Phi}_2$ are irreducible.

\begin{lem}
Let $\Phi= a(x)y^2+b(x)y+c(x)\in \C[x_1,\ldots,x_n,y]$, where $\gcd(a(x),b(x),c(x))=1$.
Then $\Phi$ is reducible if and only if $b^2-4ac$ is a square in $\C[x_1,\ldots,x_n]$.
\end{lem}

Obviously the critical set $C(f)$ is a quadric if and only if $\rank\Phi_1>0$. In this case $\overline{\Phi}_1(f)=0$ if and only if $C(f)$ has only one point at infinity, moreover $\overline{\Phi}_2(f)=0$ if and only if the projective curve $\overline{C(f)}$ is singular, i.e. $C(f)$ is a sum of two lines, finally, $\rank\Phi_2(f)=1$ if and only if $C(f)$ is a double line. In particular $O(f_1)=\Omega(2,2)\setminus (\{\overline{\Phi}_1=0\}\cup
\{\overline{\Phi}_2=0\})$, which confirms in a different way that $O(f_1)$ is an affine variety. Now we will describe the set $\{\overline{\Phi}_1=0\}\setminus \{\overline{\Phi}_2=0\}$. Let $f_2=(x^2+y,xy)$, we have:

\begin{theo}
The following conditions are equivalent:

\begin{enumerate}
\item $O(f)=\{\overline{\Phi}_1=0\}\setminus \{\overline{\Phi}_2=0\}$.
\item $C(f)$ is a parabola.
\item $O(f)=O(f_2),$ i.e. $f$ is linearly equivalent to $f_2$.
\end{enumerate}
\noindent In particular the orbit $O(f_2)$ is a smooth affine variety of dimension $11$.
\end{theo}

\begin{proof}
(1) $\Rightarrow$ (2) It is obvious.

(2) $\Rightarrow$ (3) First observe that $f=(g,h)$ is a proper mapping. Indeed, if $f$ is not proper, then  by \cite{jel1} the curve $C(f)$ must be contracted by $f$ to some point $(a,b)$. Hence $g-a,h-b$ are proportional to the equation of $C(f)$ and $J(f)$ vanishes on $\C^2$, a contradiction.

We claim that $\Delta(f)$ is not smooth. In the other case $\Delta(f)\cong \C$ and by \cite{jel} the mapping $f$ is algebraically equivalent to $(x^r,y)$. Hence $f$ must be linearly equivalent to $(x^2,y)$ and $C(f)$ is not a parabola,  a contradiction. Consequently the curve $\Delta(f)$ has a singular point. This implies that the group of linear mappings which stabilize $\Delta(f)$ has dimension at most $1$. On the other hand $f$ is not generic so the group $\Stab(f)$ in ${\mathcal GA(2,2)}$ has dimension exactly one. Hence $\dim O(f)=11=\dim\{\overline{\Phi}_1=0\}$. This means that $O(f)$ is open and dense in $\{\overline{\Phi}_1=0\}$. By the same reason $O(f_2)$ is open and dense in $\{\overline{\Phi}_1=0\}$. Since the hypersurface $\{\overline{\Phi}_1=0\}$ is irreducible, we have $O(f)\cap O(f_2)\not=\emptyset$. Hence $O(f_2)=O(f)$.

(3) $\Rightarrow$ (1) It is obvious.
\end{proof}

\begin{re}
{\rm Note that the mapping $f_2$ is a pattern of a cusp singularity.}
\end{re}

Our next aim is to describe the set $\{\overline{\Phi}_2=0\}\setminus \{\overline{\Phi}_1=0\}$. Let $f_3=(x^2+y, y^2)$. We have:

\begin{theo}
The following conditions are equivalent:

\begin{enumerate}
\item $C(f)$ is isomorphic to a cross $K_2=\{ xy=0\}$ and $\Delta(f)$ contains a parabola.
\item $O(f)$ is dense in $\{\overline{\Phi}_2=0\}\setminus \{\overline{\Phi}_1=0\}$.
\item $O(f)=O(f_3)$, i.e. $f$ is linearly equivalent to $f_3$.
\end{enumerate}

\noindent In particular the orbit $O(f_3)$ is a smooth variety of dimension $11$.
\end{theo}

\begin{proof}
(1) $\Rightarrow$ (2) Note that $f$ is a proper mapping. If $f=(g,h)$ then $f=0$ and $g=0$ have no common points at infinity.
Indeed, we can assume that $C(f)=K_2$. Hence the common zeroes of $g$ and $h$ at the line at infinity are of the form $(1:0)$ or $(0:1)$. By symmetry it is enough to consider the first case. Hence $a_1=a_2=0$ and consequently $B=0$. This implies that $J(f)\not=xy$. This contradiction shows that the curves $g=a$,$h=b$ have no common points at infinity, hence $\mu(f)=4$ and $f$ is proper.

Note that $\Delta(f)=P\cup P_1$, where $P$ is a parabola and $P_1\not=P$ is a line or a parabola. Indeed, we can assume that $P=\{y=x^2\}$. Now, if $\Delta(f)=P$ then by \cite{jel} $f$ is algebraically equivalent to $(x^r,y)$, hence $C(f)$ is not isomorphic to $K_2$, a contradiction. Denote $R=P\cap P_1$, we can assume that $R=(0,0)$.

Consider the group $\Stab(f)\subset {\mathcal GA(2,2)}$. Every affine transformation which preserves $(P,R)$ is of the form $(ax,a^2y).$ Hence $\dim\Stab(f)\le 1$ and since $f$ is not generic we have $\dim\Stab(f)= 1$. Consequently $\dim O(f)=11$ and $O(f)$ is dense in $\{\overline{\Phi}_2=0\}\setminus \{\overline{\Phi}_1=0\}$.

(2) $\Rightarrow$ (3) It is obvious.

(3) $\Rightarrow$ (1) It is obvious.
\end{proof}

\begin{re}
{\rm It is easy to see that the discriminant of $f_3$ (and hence the discriminant of every mapping from $O(f_3)$) is a union of a parabola and a line which is tangent to this parabola.}
\end{re}

Let $f_4=(x^2,y^2)$. We will describe the set $(\{\overline{\Phi}_2=0\}\setminus \{\overline{\Phi}_1=0\})\setminus O(f_2)$.

\begin{theo}
The following conditions are equivalent:

\begin{enumerate}
\item $C(f)$ and $\Delta(f)$ are isomorphic to the cross $K_2=\{ xy=0\}$.
\item $O(f)=O(f_4)$, i.e. $f$ is linearly equivalent to $f_4$.
\item $O(f)=(\{\overline{\Phi}_2=0\}\setminus \{\overline{\Phi}_1=0\})\setminus O(f_3)$.
\end{enumerate}

\noindent The orbit $O(f_4)$ is a smooth variety of dimension $10$. Moreover,
$O(f_3)\cup O(f_4)=\{\overline{\Phi}_2=0\}\setminus \{\overline{\Phi}_1=0\}$.
\end{theo}

\begin{proof}
(1) $\Rightarrow$ (2) We can assume that $C(f)=K_2$ and $\Delta(f)=K_2$. As in the proof above, the mapping $f$ is proper.
Hence by \cite{jel} the mapping $f$ is algebraically equivalent to $(x^r, y^s)$ for some $r,s>1$. Since $\mu(f)\le 4$ we have $r=s=2$. Finally $f$ is linearly equivalent to $(x^2,y^2)$ because $f\in \Omega(2,2)$.

(2) $\Rightarrow$ (3) If $f\in \{\overline{\Phi}_2=0\}\setminus \{\overline{\Phi}_1=0\}$ and $O(f)\not=O(f_3)$, then $\Delta(f)$ is isomorphic either to $K_2$ or to a line. The second possibility can be excluded. Indeed, assume that $\Delta(f)$ is a line. By \cite{jel} the mapping $f$ is algebraically equivalent to a mapping $(x^r,y)$ hence $C(f)$ is not isomorphic to $K_2$, a contradiction.

(3) $\Rightarrow$ (1) We have $f_4\in O(f)$.

\noindent Finally $\Stab(f_4)=\{ ((\alpha^{-2}x,\beta^{-2}y), (\alpha x,\beta y)), \alpha,\beta\in \C^*\}\cup \{ ((\beta^{-2}y,\alpha^{-2}x), (\alpha y,\beta x)), \alpha,\beta\in \C^*\}$,
hence $\dim\Stab(f_4)=2$ and consequently $\dim O(f_4)=10$.
 \end{proof}

Now we would like to obtain the equations of the closure of $O(f_4)$. Note that if $f=(g,h)=f_4\circ R$, where $R\in GA(2,2)$ then $g$ and $h$ are squares, i.e. the coefficients of $g$ and $h$ satisfy:

$$\rank\left[\begin{matrix} 2a_1 & b_1 & d_1\\ b_1 & 2c_1 & e_1 \\ d_1 & e_1 & 2f_1\end{matrix}\right]\leq 1\quad\text{and}\quad
\rank\left[\begin{matrix} 2a_2 & b_2 & d_2\\ b_2 & 2c_2 & e_2 \\ d_2 & e_2 & 2f_2\end{matrix}\right]\leq 1$$

Furthermore if we take $f=(g,h)=L\circ f_4\circ R$, where $(L,R)\in{\mathcal GA(2,2)}$ then $g-f_1$ and $h-f_2$ are linear combinations of two squares (both of the same two). Thus the coefficients of $f$ satisfy the condition $\rank \Psi_1\leq 2$, where:

$$\Psi_1=\left[\begin{matrix} 2a_1 & b_1 & 2a_2 & b_2\\ b_1 & 2c_1 & b_2 & 2c_2 \\ d_1 & e_1 & d_2 & e_2\end{matrix}\right].$$

Obviously $O(f_4)$ is a dense open subset of the set $\{\rank \Psi_1\leq 2\}$. In particular we obtain that $O(f_3)=\{\overline{\Phi}_2=0\}\setminus (\{\overline{\Phi}_1=0\}\cup \{\rank \Psi_1\leq 2\})$ and $O(f_4)=\{\rank \Psi_1\leq 2\}\setminus \{\overline{\Phi}_1=0\}$.

The remaining part  of orbits is contained in $ \{\overline{\Phi}_2=0\}\cap \{\overline{\Phi}_1=0\}$, in particular all of them have codimension
   greater than or equal to $2.$ The rest of the paper is devoted to description of these orbits. Note that if $f\in \{\overline{\Phi}_2=0\}\cap \{\overline{\Phi}_1=0\}$, then $C(f)$ is either two parallel lines or a double line or a line or it is the empty set or the whole $\C^2$.

Let $f_5=(x^2-x,xy)$. We will describe the set $\{\rank\Phi_1=1\}\cap\{\rank \Phi_2=2\}$.

\begin{theo}\label{th2lines}
The following conditions are equivalent:

\begin{enumerate}
\item $O(f)=\{\rank\Phi_1=1\}\cap\{\rank \Phi_2=2\}$.
\item $C(f)$ consists of two parallel lines.
\item $O(f)=O(f_5)$, i.e. $f$ is linearly equivalent to $f_5$.
\end{enumerate}

\noindent The orbit $O(f_5)$ is a smooth  variety of dimension $10$.
\end{theo}
\begin{proof}
(1) $\Rightarrow$ (2) Obvious.

(2) $\Rightarrow$ (3) By composing with suitable $R$ we may assume that $J(f)=2x^2-x$, hence $2=A=2a_1b_2-2a_2b_1$. So by composing with $L$ we may additionally assume that $a_1=1$ and $a_2=0$, hence $b_2=1$ and we may assume that $b_1=0$. Furthermore $0=B=4c_2$, $0=C=-2c_1$, $-1=D=2e_2+d_1$, $0=E=-e_1$ and $0=F=d_1e_2$. Hence either $f=(x^2+f_1,xy+d_2x-\frac{1}{2}y+f_2)$ or $f=(x^2-x,xy+d_2x+f_2)$. In both cases $f$ is equivalent to $f_5$.

(3) $\Rightarrow$ (1) In virtue of the implication (2) $\Rightarrow$ (3) it is obvious.

Finally the set $\{\rank\Phi_1=1\}\cap\{\rank \Phi_2=2\}$ is an open subset of $\{\overline{\Phi}_1=0\}\cap \{\overline{\Phi}_2=0\}.$ Since the latter set has dimension $10$ we have dim $O(f_5)=10.$
\end{proof}

Let $f_6=(x^2,xy)$. We will describe the set $\{\rank\Phi_1=1\}\cap\{\rank \Phi_2= 1\}$.

\begin{theo}
The following conditions are equivalent:

\begin{enumerate}
\item $C(f)$ is a double line.
\item $O(f)=O(f_6)$, i.e. $f$ is linearly equivalent to $f_6$.
\item $O(f)=\{\rank\Phi_1=1\}\cap\{\rank \Phi_2= 1\}$.
\end{enumerate}

\noindent The orbit $O(f_6)$ is a smooth  variety of dimension $9$.
\end{theo}

\begin{proof}
(1) $\Rightarrow$ (2) As in Theorem \ref{th2lines} we can assume that $J(f)=2x^2$, $a_1=b_1=1$ and $a_2=b_2=0$. Then it follows that $c_1=c_2=d_1=e_1=e_2=0$. Hence $f=(x^2+f_1,xy+d_2x+f_2)$ and is equivalent to $f_6$.

(2) $\Rightarrow$ (3) Note that for every $f\in\{\rank\Phi_1=1\}\cap\{\rank \Phi_2=1\}$ the set $C(f)$ is a double line. Hence by
the first part of the proof, we have  $\{\rank\Phi_1=1\}\cap\{\rank \Phi_2=1\}=O(f_6)$.

(3) $\Rightarrow$ (1) Obvious.

Finally note that $\Stab(f_6)=\{ (\alpha^{-2}x, (\alpha\gamma)^{-1}y-\beta\gamma^{-1}\alpha^{-2}x), (\alpha x, \beta x+ \gamma y), \alpha, \gamma \in \C^*, \beta\in \C\}.$  Hence we have dim $O(f_6)=9.$
\end{proof}

From the theorems above it follows that $\{\rank\Phi_1>0\}=O(f_1)\cup\ldots\cup O(f_6)$. We will now focus on the mappings satisfying $\rank\Phi_1(f)=0$. Let

$$\Phi_3(f)=\left[\begin{matrix} a_1 & b_1 & c_1 \\ a_2 & b_2 & c_2 \end{matrix}\right]\quad\text{and}\quad
\Phi_4(f)=\left[\begin{matrix} a_1 & b_1 & c_1 & d_1 & e_1\\ a_2 & b_2 & c_2 & d_2 & e_2\end{matrix}\right].$$

The assumption $\rank\Phi_1(f)=0$ is equivalent to $\rank\Phi_3(f)<2$ and to $C(f)$ not being a quadric. In particular $O(f)\cap\Omega(2,1)\neq\emptyset$, so from now on we will assume that $f\in\Omega(2,1)$, i.e. $a_2=b_2=c_2=0$.

Let $f=(g_2+g_1+g_0,h_1+h_0)$, where $g_i$ and $h_i$ are homogeneous of degree $i$. We may assume that $\deg g\geq\deg h$ then $\rank\Phi_3(f)=0$ if and only if $g_2=0$, $\rank\Phi_4(f)<2$ if and only if $h_1=0$ and $\rank\Phi_4(f)=0$ if and only if $f$ is constant.

Now we will describe the case when $g_2h_1\neq 0$. Let $f_7=(xy,x+y)$, $f_8=(x,xy)$, $f_9=(x^2,y)$, $f_{10}=(x^2+y,x)$ and $f_{11}=(x^2,x)$.

\begin{theo}
Let $f=(g_2+g_1+g_0,h_1+h_0)$ and $g_2h_1\neq 0$, then one of the following holds:

\begin{enumerate}
\item $g_2$ is not a square, $h_1$ does not divide $g_2$. Then $f\in O(f_7)$. Moreover $C(f)$ is a line, $\Delta(f)$ is a parabola and $\dim\Stab(f)=2$. The mapping $f$ is proper.
\item $g_2$ is not a square, $h_1$ divides $g_2$. Then $f\in O(f_8)$. Moreover $C(f)$ is a line and $\Delta(f)$ is a line and $\dim\Stab(f)=3$. The mapping $f$ is not proper.
\item $g_2$ is a square, $h_1$ does not divide $g_2$. Then $f\in O(f_9)$. Moreover $C(f)$ and $\Delta(f)$ are lines and $\dim\Stab(f)=3$. The mapping $f$ is proper.
\item $g_2$ is a square, $h_1$ divides $g_2$ but not $g_1$. Then $f\in O(f_{10})$. Moreover $f$ is an automorphism and $\dim\Stab(f)=4$. The mapping $f$ is proper.
\item $g_2$ is a square, $h_1$ divides $g_2$ and $g_1$. Then $f\in O(f_{11})$. Moreover $C(f)=\C^2$, $\im(f)$ is a parabola and $\dim\Stab(f)=5$. The mapping $f$ is not dominant.
\end{enumerate}
\end{theo}
\begin{proof}
(1) Since $g_2$ is not a square it is a product of two independent linear
forms and we may assume, by composing with $R$, that $g_2=xy$. Moreover
since $h_1$ does not divide $g_2$ we have $h_1=d_2x+e_2y$ for
$d_2,e_2\in\C^*$. Composing with $R(x,y)=(x/d_2,y/e_2)$ and
$L(x,y)=(d_2e_2x,y)$ we obtain $h_1=x+y$ and retain $g_2=xy$. Thus
$f=(xy+d_1x+e_1y+f_1,x+y+f_2)$ and for $R(x,y)=(x-e_1,y-d_1)$ and
$L(x,y)=(x+e_1d_1-f_1,y+d_1+e_1-f_2)$ we obtain $f_7=L\circ f\circ R$.
Clearly $C(f_7)=\{y=x\}$ is a line and $\Delta(f_7)=\{4x=y^2\}$ is a parabola. It
remains to calculate $\dim\Stab(f_7)$.  Since for every $(R,L)\in\Stab(f_7)$ the mapping $R$ must preserve the parabola $\Delta(f)$ we see that $\dim\Stab(f_7)\le 2$. On the other hand the codimension of $O(f_7)$ is a least two. Consequently $\dim\Stab(f_7) = 2$.

(2) Similarly as in (1) we may assume that $g_2=xy$ and $h_1=x$. It follows that $f$ is equivalent with $f_8$. Moreover $J(f_8)=x$ and $\Delta(f_8)=\{ x=0\}$. Let $(L,R)\in\Stab(f_8)$ and $R=(R_1,R_2)$. Since $R$ preserves $C(f_8)$ we have $R_1=\alpha x$  and $R_2=\beta y+\gamma$, where $\alpha,\beta\in\C^*$ and $\gamma\in\C$.
Thus $f_8\circ R (x,y)=(\alpha\beta xy+\alpha\gamma,\alpha x)$ and for
$L(x,y)=((x-\alpha\gamma)/\alpha\beta,y/\alpha)$ we obtain $L\circ f_8\circ
R=f_8$.

(3) We may assume that $g_2=x^2$ and $h_1=y$. It follows that $f$ is equivalent with $f_9$. Moreover $J(f_9)=2x$ and $\Delta(f_9)=\{x=0\}$. As above, for $(L,R)\in\Stab(f_9)$ and $R=(R_1,R_2)$ we obtain $R_1=\alpha x$  and $R_2=\beta y+\gamma$, where $\alpha,\beta\in\C^*$ and $\gamma\in\C$.

(4) We may assume that $g_2=x^2$, $g_1=y$ and $h_1=x$. It follows that $f$ is equivalent with $f_{10}$. Let $(L,R)\in\Stab(f_{10})$ and $R=(R_1,R_2)$. Since $R_1^2$ is the only term in $f_{10}\circ R$ with quadratic part we have $R_1=\alpha x+\beta$, where $\alpha\in\C^*$ and $\beta\in\C$. Moreover the coefficients at $x^2$ and $y$ in the first component of $f_{10}\circ R$ must be equal, thus $R_2=\gamma x+\alpha^2 y+\delta$, where $\gamma,\delta\in\C$.

(5) Similarly as in (1) we may assume that $g_2=x^2$, $g_1=x$ and $h_1=x$. It follows that $f$ is equivalent with $f_{11}$. Let $(L,R)\in\Stab(f_{11})$ and $R=(R_1,R_2)$. As in (4) we have $R_1=\alpha x+\beta$, where $\alpha\in\C^*$ and $\beta\in\C$, however $R_2$ is arbitrary.
\end{proof}

Note that the conditions ``$g_2$ is a square'', ``$h_1$ divides $g_2$'' and ``$h_1$ divides $g_1$'' can be expressed in terms of the coefficients of $f\in\Omega(2,2)$. We will do it in the affine chart $a_1\neq 0$ -- one of the six natural charts covering the set $\{g_2\neq 0\}$.

Obviously $g_2$ is a square if and only if $4a_1c_1-b_1^2=0$. Moreover $h_1=d_2x+e_2y-\frac{a_2}{a_1}(d_1x+e_1y)$, hence it divides $g_1$ if and only if $a_1(a_1e_2-a_2e_1)^2-b_1(a_1e_2-a_2e_1)(a_1d_2-a_2d_1)+c_1(a_1d_2-a_2d_1)^2=0$. Finally, $h_1$ divides $g_1$ if and only if $d_1(a_1e_2-a_2e_1)-e_1(a_1d_2-a_2d_1)=0$.

Notice that from $g_2=0$ and $h_1\neq 0$ follows that $f\in GA(2,2)$. In this case $f\in O(f_{12})$, where $f_{12}=(x,y)$, and $\dim O(f)=6$.

The last case to describe is $h_1=0$. Let $f_{13}=(xy,0)$, $f_{14}=(x^2+y,0)$, $f_{15}=(x^2,0)$, $f_{16}=(x,0)$ and $f_{17}=(0,0)$.

\begin{theo}
Let $f=(g_2+g_1+g_0,0)$, then $C(f)=\C^2$ and one of the following holds:

\begin{enumerate}
\item $g_2$ is not a square. Then $f\in O(f_{13})$ and $\dim\Stab(f)=4$.
\item $g_2$ is a square, $g_1$ is nonzero and does not divide $g_2$. Then $f\in O(f_{14})$ and $\dim\Stab(f)=5$.
\item $g_2$ is a square, $g_1$ is zero or divides $g_2$. Then $f\in O(f_{15})$ and $\dim\Stab(f)=6$.
\item $g_2=0$ and $g_1\neq 0$. Then $f\in O(f_{16})$ and $\dim\Stab(f)=7$.
\item $g=0$. Then $f\in O(f_{17})$ and $\dim\Stab(f)=10$.
\end{enumerate}
\end{theo}
\begin{proof}
Note that in cases (1)--(4) the image of $f$ is a line. Thus if $(L,R)\in\Stab(f)$ then $R$ determines $L$ up to an automorphism preserving every point of the line, the group of such automorphisms is two dimensional. In case (5) the image of $f$ is a point, which is preserved by a four dimensional group.

(1) Since $g_2$ is not a square we may assume, by composing with $R$, that $g_2=xy$. Thus $f=(xy+d_1x+e_1y+f_1,0)$, which is equivalent with $f_{13}$. Let $(L,R)\in\Stab(f_{13})$ and $R=(R_1,R_2)$. We have $R_1R_2=\alpha_1xy+\alpha_2$, consequently $R$ is equal either to $(\beta_1 x,\beta_2 y)$ or to $(\beta_1 y,\beta_2 x)$, where $\beta_1,\beta_2\in\C^*$.

(2) We may assume that $g_2=x^2$ and $g_1=y$, so $f$ is equivalent with $f_{14}$. Let $(L,R)\in\Stab(f_{14})$ and $R=(R_1,R_2)$. The quadratic part of $R_1^2$ is equal $\alpha_1^2x^2$, thus $R=(\alpha_1 x+ \alpha_2, \alpha_1^2 y-2\alpha_1\alpha_2 x+\alpha_3)$, where $\alpha_1\in\C^*$ and $\alpha_2,\alpha_3\in\C$.

(3) We may assume that $f=(x^2+d_1x+f_1,0)$, which is equivalent with $f_{15}$. Let $(L,R)\in\Stab(f_{15})$ and $R=(R_1,R_2)$. We have $R_1=\alpha_x$ and $R_2$ arbitrary.

(4) and (5) are obvious.

\end{proof}

\section{The real case}

In this section we will determine the orbits in $\Omega(2,2)$ over the field of real numbers. We will base on the complex case and to avoid confusion we will use the subscripts $\mathbb{R}$ and $\mathbb{C}$ to indicate over which field an object is considered. Obviously for $f\in\Omega_{\mathbb{R}}(2,2)$ we have $O_{\mathbb{R}}(f)\subset O_{\mathbb{C}}(f)\cap\Omega_{\mathbb{R}}(2,2)$. In most cases equality holds, however there are two cases when the restriction of a complex orbit splits into two real orbits. We have the following possibilities:

(1) the restriction of the generic complex orbit splits into two semialgebraic orbits given by $\overline{\Phi}_1(f)<0$ and $\overline{\Phi}_1(f)>0$. The first case is represented by $f_1=\left(x^2+y,y^2+x\right)$ with $C(f_1)=\{4xy-1=0\}$ a hyperbola and $\Delta(f_1)=\left\{2^8x^2y^2-2^8x^3-2^8y^3+2^5\cdot 9xy-27=0\right\}$ a reduced and irreducible curve with a cusp at $f_1\left(\frac{1}{2},\frac{1}{2}\right)$. The second case is represented by $f_{1'}=\left(x^2-y^2+x,2xy-y\right)$ with $C(f_{1'})=\left\{4x^2+4y^2-1=0\right\}$ a circle and $\Delta(f_{1'})=\{2^{16}(x^2+y^2)^2+2^{17}(-x^3+3xy^2)+2^9\cdot 3^3\cdot 5(x^2+y^2)-(15)^3\}$ a reduced and irreducible curve with $3$ cusps at points $f_{1'}\left(\frac{1}{4},\frac{-\sqrt{3}}{4}\right)$, $f_{1'}\left(\frac{1}{4},\frac{\sqrt{3}}{4}\right)$ and $f_{1'}\left(\frac{1}{2},0\right)$.

(2) $O_{\mathbb{R}}(f_2)=O_{\mathbb{C}}(f_2)\cap\Omega_{\mathbb{R}}(2,2)$.

(3)  $O_{\mathbb{R}}(f_3)=O_{\mathbb{C}}(f_3)\cap\Omega_{\mathbb{R}}(2,2)$.

(4) the restriction of complex orbit of $f_4=\left(x^2,y^2\right)$ splits into two orbits given by $\overline{\Phi}_1(f)<0$ and $\overline{\Phi}_1(f)>0$. The first case is represented by $f_4$ with $C(f_4)=\{4xy=0\}$ two intersecting lines and $\Delta(f_4)=\{xy=0\}$ also two intersecting lines. The second case is represented by $f_{4'}=\left(x^2-y^2,xy\right)$ with $C(f_{1'})=\left\{2x^2+2y^2=0\right\}$ a point.

(5) for $f\in\{f_5,\ldots,f_{17}\}$ we have $O_{\mathbb{R}}(f)=O_{\mathbb{C}}(f)\cap\Omega_{\mathbb{R}}(2,2)$.

\begin{proof}
(1)
Let $f\in O_{\mathbb{C}}(f_1)\cap\Omega_{\mathbb{R}}(2,2)$. Note that $f$ has $3$ cusps. Since the cusps are given by real equations they are either real or in complex conjugate pairs. So $f$ may have either $3$ or $1$ real cusp. Note that $f_{1'}=\left(x^2-y^2+x,2xy-y\right)$ has $3$ real cusps and $f_1$ has $1$.

Assume that $f$ has $3$ real cusps. There are some $L,R\in GA_{\mathbb{C}}(2,2)$ such that $f_{1'}=L\circ f\circ R$. Note that $R$ maps the three real and non-collinear cusps of $f_{1'}$ into the three real cusps of $f$. Hence $R(\mathbb{R}^2)=\mathbb{R}^2$, so $R\in GA_{\mathbb{R}}(2,2)$. Furthermore $f\circ R(\mathbb{R}^2)=\mathbb{R}^2$, hence also $L\in GA_{\mathbb{R}}(2,2)$. Thus $f\in O_{\mathbb{R}}(f_{1'})$.

Now assume that $f$ has $1$ real cusp. There are some $L,R\in GA_{\mathbb{C}}(2,2)$ such that $f_1=L\circ f\circ R$. By composing with an element of $\Stab(f_1)$ (see the proof of Theorem \ref{th1}) we may assume that $R$ maps the real cusp of $f_1$ to the real cusp of $f$. Since the complex cusps are conjugate the mapping $\overline{R}$, i.e. the conjugate of $R$, coincides with $R$ on the three non-collinear cusps. Hence $\overline{R}=R$, so $R\in GA_{\mathbb{R}}(2,2)$ and consequently $L\in GA_{\mathbb{R}}(2,2)$ and $f\in O_{\mathbb{R}}(f_1)$.

Now we will show that if $f'\in O_{\mathbb{R}}(f)$ then $\sgn(\overline{\Phi}_1(f'))=\sgn(\overline{\Phi}_1(f))$. Indeed, let $L,R\in GA_{\mathbb{R}}(2,2)$, we have $J(L\circ f\circ R)=(J(L)\circ f\circ R)\cdot(J(f)\circ R)\cdot J(R)=J(L)\cdot J(R)\cdot (J(f)\circ R)$. Note that $\Phi_1(f)$ is the matrix defining the quadratic form which is the homogeneous part of $2J(f)$ of degree $2$. Consequently $\Phi_1(L\circ f\circ R)=J(L)\cdot J(R)\cdot D(R)^T\cdot \Phi_1(f)\cdot D(R)$, where $D(R)$ is the derivative of the affine map $R$. Hence $\overline{\Phi}_1(L\circ f\circ R)=J(L)^2\cdot J(R)^2\cdot\det(D(R)^T)\cdot\det(\Phi_1(f))\cdot\det(D(R))=J(L)^2\cdot J(R)^4\cdot\overline{\Phi}_1(f)$. So in particular $\sgn(\overline{\Phi}_1(L\circ f\circ R))=\sgn(\overline{\Phi}_1(f))$.

As a consequence we obtain that $f\in O_{\mathbb{R}}(f_{1'})$ if and only if $\overline{\Phi}_1(f)>0$ and $f\in O_{\mathbb{R}}(f_1)$ if and only if $\overline{\Phi}_1(f)<0$. Indeed if $f\in O_{\mathbb{C}}(f_1)\cap\Omega_{\mathbb{R}}(2,2)$ then either $f\in O_{\mathbb{R}}(f_{1})$ and $\sgn(\overline{\Phi}_1(f))=\sgn(\overline{\Phi}_1(f_{1}))=-1$ or $f\in O_{\mathbb{R}}(f_{1'})$ and $\sgn(\overline{\Phi}_1(f))=\sgn(\overline{\Phi}_1(f_{1'}))=1$.

(2)
Let $f\in O_{\mathbb{C}}(f_2)\cap\Omega_{\mathbb{R}}(2,2)$. Note that $C_\mathbb{C}(f)$ is a parabola, hence $C_\mathbb{R}(f)$ is also a parabola. Moreover $f$ has one cusp over $\mathbb{C}$ and since it is given by real equations the cusp must be real. Let $R_0\in GA_{\mathbb{R}}(2,2)$ be such that $C_\mathbb{R}(f\circ R_0)=\{2x^2=y\}_\mathbb{R}=C_\mathbb{R}(f_2)$ and $f\circ R_0$ has the cusp at origin. Moreover let $L,R\in GA_{\mathbb{C}}(2,2)$ be such that $f\circ R_0=L\circ f_2\circ R$. Note $R$ is an automorphism of $C_\mathbb{C}(f\circ R_0)=\{2x^2=y\}_\mathbb{C}=C_\mathbb{C}(f_2)$ and preserves the origin. Hence $R=(ax,a^2y)$ for some $a\in\mathbb{C}^*$. Let $L_1=(a^2x,a^3y)$, note that we have $f_2\circ R=L_1\circ f_2$ and thus $f=L\circ L_1\circ f_2\circ R_0^{-1}$. Since the mapping $L\circ L_1$ maps $\mathbb{R}^2$ onto itself, we have $L\circ L_1\in GA_{\mathbb{R}}(2,2)$ and finally $f\in O_{\mathbb{R}}(f_2)$.

(3)
Let $f\in O_{\mathbb{C}}(f_3)\cap\Omega_{\mathbb{R}}(2,2)$. Note that $C_\mathbb{C}(f)$ consists of two distinct intersecting lines, say $\{\alpha(x,y)=0\}$ and $\{\beta(x,y)=0\}$. Since $J(f)=\alpha\beta$ has real coefficients we may assume that $\alpha$ and $\beta$ either have real  or complex conjugate coefficients. However, the latter case can be excluded. Indeed, assume that $\alpha$ and $\beta$ have complex conjugate coefficients. Since $f$ has real coefficients $f(\overline{x},\overline{y})=\overline{f(x,y)}$ and it follows that the components of the discriminant $\Delta(f)$ are also conjugate. This is a contradiction since the discriminant of mappings equivalent to $f_3$ consists of a line and a parabola.

So $C_\mathbb{R}(f)$ is a cross and $\Delta_\mathbb{R}(f)$ is a sum of a parabola and a tangent line. Let $L_0,R_0\in GA_{\mathbb{R}}(2,2)$ be such that $C_\mathbb{R}(L_0\circ f\circ R_0)=\{xy=0\}_\mathbb{R}=C_\mathbb{R}(f_3)$ and $\Delta_\mathbb{R}(L_0\circ f\circ R_0)=\{y(y-x^2)=0\}_\mathbb{R}=\Delta_\mathbb{R}(f_3)$. Moreover let $L,R\in GA_{\mathbb{C}}(2,2)$ be such that $L_0\circ f\circ R_0=L\circ f_3\circ R$. Note that $R$ preserves the set $\{xy=0\}_\mathbb{C}$ and $L$ preserves the set $\{y(y-x^2)=0\}_\mathbb{C}$, hence $L=(ax,a^2y)$ and $R=(bx,cy)$ for some $a,b,c\in\mathbb{C}^*$. Since $L\circ f_3\circ R=(ab^2x^2+acy,a^2c^2y^2)$ is a real mapping we have $ac,ab^2\in\mathbb{R}$. Moreover $f_3=(b^2x,b^4y)\circ f_3\circ(b^{-1}x,b^{-2}y)$, so $L_0\circ f\circ R_0=(ab^2x,a^2b^4y)\circ f_3\circ(x,cb^{-2}y)$ and $f\in O_{\mathbb{R}}(f_3)$.

Note that $f\in O_{\mathbb{R}}(f_3)$ if and only if $\sgn(\overline{\Phi}_1(f))=\sgn(\overline{\Phi}_1(f_3))=-1$. Thus we have shown that for $f\in\Omega_{\mathbb{R}}(2,2)$ the conditions $\overline{\Phi}_1(f)\neq 0$, $\overline{\Phi}_2(f)=0$ and $\rank\Psi_1(f)=3$ imply $\overline{\Phi}_1(f)<0$.

(4)
Let $f\in O_{\mathbb{C}}(f_4)\cap\Omega_{\mathbb{R}}(2,2)$. Similarly as in (3) the set $C_\mathbb{C}(f)$ consists of two lines $\{\alpha(x,y)=0\}$ and $\{\beta(x,y)=0\}$ and we may assume that $\alpha$ and $\beta$ either have real or complex conjugate coefficients.

If $C_\mathbb{R}(f)$ is a cross then we proceed similarly as in (3). We have $R_0\in GA_{\mathbb{R}}(2,2)$ and $L,R\in GA_{\mathbb{C}}(2,2)$ such that $C_\mathbb{R}(f\circ R_0)=\{xy=0\}_\mathbb{R}$. Since $R$ preserves the set $\{xy=0\}_\mathbb{C}$ and $f_4=(a^{-2}x,b^{-2}y)\circ f_4\circ (ax,by)$ we may assume that $R=(x,y)$ and deduce that $f\in O_{\mathbb{R}}(f_4)$.

Now let $\alpha$ and $\beta$ have complex conjugate coefficients, we will show that $f\in O_{\mathbb{R}}(f_{4'})$, where $f_{4'}=(x^2-y^2,xy)$.
Indeed, let $R_0\in GA_{\mathbb{R}}(2,2)$ be such that $R_0(0,0)$ is the point of intersection of $\{\alpha(x,y)=0\}$ and $\{\beta(x,y)=0\}$ and $R_0(\{y-ix=0\})=\{\alpha(x,y)=0\}$. It follows that $R_0(\{y+ix=0\})=\{\beta(x,y)=0\}$ and consequently $C_\mathbb{C}(f\circ R_0)=\{x^2+y^2=0\}_\mathbb{C}=C_\mathbb{C}(f_{4'})$. Let $L,R\in GA_{\mathbb{C}}(2,2)$ be such that $f\circ R_0=L\circ f_{4'}\circ R$. Note that $R$ preserves $\{x^2+y^2=0\}_\mathbb{C}$ and we may assume that it does not switch the lines $\{y-ix=0\}$ and $\{y-ix=0\}$. So $R=(ax+by,-bx+ay)$ for some $a,b\in\mathbb{C}$ such that $a^2+b^2\neq 0$. Let $L_1=((a^2-b^2)x+4aby,-abx-(a^2-b^2)y)$, note that $L_1$ is an automorphism and $L_1\circ f_{4'}=f_{4'}\circ R$. Thus $f\circ R_0=L\circ L_1\circ f_{4'}$, so $L\circ L_1\in GA_{\mathbb{R}}(2,2)$ and $f\in O_{\mathbb{R}}(f_{4'})$.

As in (1) we obtain that $f\in O_{\mathbb{R}}(f_4)$ if and only if $\overline{\Phi}_1(f)<0$, and $f\in O_{\mathbb{R}}(f_{4'})$ if and only if $\overline{\Phi}_1(f)>0$.

(5)
The description of orbits of $f_5$--$f_{17}$ provided in Section \ref{secMR} can be done in exactly the same way over $\mathbb{R}$.
\end{proof}

\section{Final Remarks}

At the end of this paper we state two conjectures:

\vspace{5mm}

{\bf Conjecture 1}. {\it For every $d_1, d_2>0$ the set $U$ of generic mappings in $\Omega(d_1,d_2)$ is an open affine subvariety of $\Omega(d_1,d_2)$. In particular every generic mapping is topologically stable, i.e. remains generic after a small deformation.}

\vspace{5mm}

{\bf Conjecture 2}. {\it For every $d_1, d_2>0$ a mapping $f\in \Omega(d_1,d_2)$ is generic if $\Delta(f)$ contains only cusps and nodes and the number of cusps and nodes is maximal possible (see \cite{fjr}).}

\vspace{5mm}

\noindent Note that for $d_1=d_2=2$ these Conjectures are true, which follows from our paper.

\end{document}